\documentclass[a4paper,12pt,intlimits,oneside]{amsart}

\usepackage{enumerate}

\usepackage{latexsym,amssymb}

\newcommand{\comment}[1]{}

\newcommand{\bC}{{\mathbb C}}

\newcommand{\bR}{{\mathbb R}}

\def\a{{\mathfrak a}}
\def\n{{\texttt n}}

\def\H{{\mathcal H}}

\def\X{{\mathcal X}}

\def\bmo{\mathfrak{bmo}}
\def\BMO{\rm{B\! M\! O}}

\def\ExpL{\rm{E\!x\!p\; L}}
\def\LlogL{\rm{L\;l\!o\!g\;L}}

\def\lsim{\raisebox{-1ex}{$~\stackrel{\textstyle <}{\sim}~$}}

\newcounter{rea}
\newcounter{rej}
\newcounter{res}
\newtheorem{thm}{Theorem}[section]
\newtheorem{prop}[thm]{Proposition}

\newtheorem{lem}[thm]{Lemma}
\newtheorem{defn}[thm]{Definition}
\newtheorem{remark}[thm]{Remark}

\begin{document}

\title[]{Products of functions in $\BMO$ and $\H^{1}$ spaces on spaces of homogeneous type}

\author[J. Feuto]{Justin Feuto}
\address{Laboratoire de Math\'ematiques Fondamentales, UFR Math\'ematiques et Informatique, Universit\'e de Cocody, 22 B.P 1194 Abidjan 22. Côte d'Ivoire}
\email{{\tt justfeuto@yahoo.fr}}

\comment{\subjclass{32A37 47B35 47B10 46E22}}
\keywords{space of homogeneous type, Hardy-Orlicz spaces, atomic decomposition, space of test function,  distribution space, maximal function.}

\begin{abstract}
We give an extension to certain \textit{RD-space} $\X$, i.e space of homogeneous type in the sense of Coifman and Weiss, which has the reverse doubling property, of the definition and various properties of the product of functions in $\BMO(\X)$ and $\H^{1}(\X)$, and functions in Lipschitz space $\Lambda_{\frac{1}{p}-1}(\X)$ and $\H^{p}(\X)$ for $p\in\left(\frac{\n}{\n+\theta},1\right]$, where $\n$ and $\theta$ denote respectively the "dimension" and the order of $\X$.

\end{abstract}
\maketitle

\section{Introduction}
It is well known that ${\BMO}(\mathbb R^{n})$ is the dual space of $\H^{1}(\mathbb R^{n})$ and that multiplication by $\varphi\in\mathcal D(\mathbb R^{n})$ is a bounded operator on ${\BMO}(\mathbb R^{n})$. Those facts allow  Bonami, Iwaniec, Jones and  Zinsmeister, to define in \cite{BIJZ} a product $\mathfrak b\times \mathfrak h$ of $\mathfrak b\in {\BMO}(\bR^{n})$ and $\mathfrak h\in\H^{1}(\bR^{n})$ as a distribution, operating on a test function $\varphi\in\mathcal D(\bR^{n})$ by the rule
\begin{equation}
\left\langle \mathfrak b\times\mathfrak h,\varphi\right\rangle:=\left\langle \mathfrak b\varphi,\mathfrak h\right\rangle.\label{product}
\end{equation}
 They  proved that such distributions are sums of a function in $L^{1}(\bR^{n})$ and a distribution in a Hardy-Orlicz space $\H^{\wp}(\bR^{n},\nu)$ where 
\begin{equation}
\wp(t)=\frac{t}{\log(e+t)} \text{ and } d\nu(x)=\frac{dx}{\log(e+\left|x\right|)}.
\end{equation}
The idea of defining the above product is motivated among other things by the fact that for $1<p<\infty$, the product $fg$ of  $f\in L^{p}(\bR^{n})$ and $g$ in the dual space $L^{p'}(\bR^{n})$ of $L^{p}(\bR^{n})$ is integrable (consequently is a distribution). The Hardy space $\H^{1}(\mathbb R^{n})$ being the right substitute of $L^{1}(\bR^{n})$ in many problems, it seems natural to look at its product with its dual space  ${\BMO}(\bR^{n})$. Following of the idea in \cite{BIJZ}, A. Bonami and J. Feuto in \cite{BF} extend results, replacing ${\BMO}(\bR^{n})$ by $\bmo(\bR^{n})$, defined as the space of  locally integrable functions $\mathfrak b$ such that
\begin{equation}\label{bmo}
\sup_{|B|\leq 1}\left( \frac 1 {|B|}\int_B
|\mathfrak b(x)-\mathfrak b_B|dx\right)<\infty\ \ \ \ \mbox{\rm and } \sup_{|B|\geq
1}\left( \frac 1 {|B|}\int_B |\mathfrak b(x)|dx\right)<\infty,
\end{equation}
where $B$ varies among all balls of $\bR^n$, $|B|$ denotes the
measure of the ball $B$ and $\mathfrak b_B$ is the mean of $\mathfrak b$ on $B$. They proved that in this case, the weight $x\mapsto\frac{dx}{\log(e+\left|x\right|)}$ is not necessary.

They also proved that for $\mathfrak h$ in the Hardy space $\H^{p}(\bR^{n})$ $(0<p<1)$ the Hardy-Orlicz space is replaced  by $\H^{p}(\bR^{n})$ provided $\mathfrak b$ belongs to the inhomogeneous Lipschitz space $\Lambda_{n(\frac{1}{p}-1)}(\bR^{n})$. 

The space of homogeneous type introduced by R.R Coifman and G. Weiss in \cite{CW1} being the right space for generalize results stated in the euclidean spaces, we give here the analogous of those results in this context. For this purpose, we consider a space of homogeneous type $(\X,d,\mu)$ (see Section 2 for more explanation about this space) in which all annuli are not empty, i.e. $B(x,R)\setminus B(x,r)\neq\emptyset$ for all $x\in\X$ and $0<r<R<\infty$, where $B(x,r)=\left\{y\in\X:d(x,y)<r\right\}$ is the ball centered at $x$ and with radius $r$. According to \cite{W}, the doubling measure $\mu$ then satisfies the reverse doubling property: there exist two positive constants $\kappa$ and a constant $c_{\mu}$ depending only on $\mu$, such that
\begin{equation}
\frac{\mu(B)}{\mu(\tilde{B})}\geq c_{\mu}\left(\frac{r(B)}{r(\tilde{B})}\right)^{\kappa}\text{ for all balls }\tilde{B}\subset B,\label{reversedoubling}
\end{equation}
where $r(B)$ denotes the radius of the ball $B$. This reverse doubling condition yields that $\mu(\X)=\infty$. Using the doubling condition (\ref{doublingcondition}) and the reverse condition (\ref{reversedoubling}), we have that
\begin{equation} c_{\mu}\lambda^{\kappa}\mu\left(B(x,r)\right)\leq\mu\left(B(x,\lambda r)\right)\leq C_{\mu}\lambda^{\texttt{n}}\mu\left(B(x,r)\right)\label{RD}
\end{equation}
for all $x\in\X$, $r>0$ and $\lambda\geq 1$. We will refer to \texttt{n} as the dimension of the space.
We will also assume that there exists a positive non decreasing function $\varphi$ defined on $\left[0,\infty\right)$ such that for all $x\in\X$ and $r>0$, 
\begin{equation}
\mu\left(B_{(x,r)}\right)\sim \varphi(r).\label{C1}\footnote{Hereafter we propose the following abbreviation $A\sim B$ for the inequalities $C^{-1}A\leq B\leq CB$, where $C$ is a positive constant not depending on not depending on the main parameters.}
\end{equation}
Notice that (\ref{reversedoubling}),(\ref{doublingcondition}) and (\ref{C1}) imply that 
\begin{equation}
r^{\n}\lsim\footnote{$A\lsim B$ mean the ratio $A/B$ is bounded away from zero by a constant independent of the relevant variables in $A$ and $B$}\varphi(r)\lsim r^{\kappa}\text{ if } 0<r<1\label{norminf}
\end{equation}
and
\begin{equation}
r^{\kappa}\lsim\varphi(r)\lsim r^{\n}\text{ if } 1\leq r.\label{normsup}
\end{equation}
These spaces are particular case of the class spaces of homogeneous type named \textit{RD-spaces} in \cite{GLY}. An example of such space is obtained by considering a Lie group $X$ with polynomial growth equipped with a left  Haar measure $\mu$ and the Carnot-Carath\'eodory metric $d$ associated with a H\"ormander system of left invariant vector fields (see \cite{HMY},\cite{Ma} and \cite{Va}).

We use the maximal characterization of Hardy spaces in space of homogeneous type as developed by Grafakos, Lu and Yang in \cite{GLY}. It is proved that this maximal characterization of $\H^{p}(\X,d,\mu)$ agrees with the atomic characterization of Coifman and Weiss in \cite{CW2} if  $p\in\left(\frac{\n}{\n+\theta},1\right]$, where $\theta$ is as in relation (\ref{order}).

We recall that for $p\in\left(0,1\right]$ and $q\in\left[1,\infty\right]\cap\left(p,\infty\right]$, a function $\a\in L^{q}(\X,d,\mu)$ is said to be a $(p,q)$-atom if the following conditions are fulfilled:
\begin{enumerate}
	\item [(a1)] $\a$ is supported in a ball $B$,
	\item[(a2)] $\left\|\a\right\|_{L^{q}(\mathcal X,d,\mu)}\leq\left[\mu(B)\right]^{\frac{1}{q}-\frac{1}{p}}\text { if   }q<\infty$ \\
	and $\left\|\a\right\|_{L^{\infty}(\X,d,\mu)}\leq\mu(B)^{-\frac{1}{p}}\text{ if }q=\infty,$
	\item[(a3)] $\int_{\X}\a(x)d\mu(x)=0.$ 
\end{enumerate}

It is proved in Corollary 4.19 of \cite{GLY} that for $p\in\left(\frac{\n}{\n+\theta},1\right]$ and $q\in\left(p,\infty\right]\cap\left[1,\infty\right]$, $f\in\H^{p}(\X,d, \mu)$ if and only if there is a sequence $(\a_{i})_{i\geq 0}$ of $(p,q)$-atoms, each $\a_{i}$ supported in a ball $B_{i}$, and a sequence $(\lambda_{i})_{i\geq 0}$ of scalars such that
\begin{equation}
\mathfrak h=\sum^{\infty}_{i=1}\lambda_{i}\a_{i}\text{ and } \sum^{\infty}_{i=1}\left|\lambda_{i}\right|^{p}<\infty,\label{atomdecompo}
\end{equation}
where the first series is considered in the sense of distribution as defined in \cite{GLY}, and $\left\|\mathfrak h\right\|_{\H^{p}(\X)}\sim\inf\left\{\left(\sum_{i\geq 0}\left|\lambda_{i}\right|^{p}\right)^{\frac{1}{p}}\right\}$, the infimum being taken over all the decomposition of $f$ as above and $\left\|\mathfrak h\right\|_{\H^{p}(\X)}$ as in (\ref{normhp}). For  $\mathfrak b\in {\BMO}(\X,d,\mu)$ and $\mathfrak h\in\H^{1}(\X,d,\mu)$ as in (\ref{atomdecompo}), the series $\sum^{\infty}_{i=1}\lambda_{i}(\mathfrak b-\mathfrak b_{B_{i}})\a_{i}$ and $\sum^{\infty}_{i=1}\lambda_{i}\mathfrak b_{B_{i}}\a_{i}$ converge in the sense of distribution as we can see in the proof of Theorem \ref{main}. Thus we define the product of $\mathfrak b\times\mathfrak h$ as the sum of both series, i. e. we put
 \begin{equation}
 \mathfrak b\times\mathfrak h:=\sum^{\infty}_{i=1}\lambda_{i}(\mathfrak b-\mathfrak b_{B_{i}})\a_{i}+\sum^{\infty}_{i=1}\lambda_{i}\mathfrak b_{B_{i}}\a_{i}.
 \end{equation}
 Our main result can be stated as follows.

\begin{thm}\label{main}
For $\mathfrak h\in\mathcal H^1(\X,d,\mu)$ and $\mathfrak b\in {\BMO}(\X,d,\mu)$,
the product $\mathfrak b\times \mathfrak h$ can be given a meaning in the sense of
distributions. Moreover, if $x_{0}$ is a fixed element of $\X$ then we have the inclusion
\begin{equation}\label{inclusion}
  \mathfrak b\times \mathfrak h\in L^1(\X,d,\mu)+ \mathcal H^{\wp}(\X,d,\nu),
  \end{equation}
  where 
  \begin{equation}
  d\nu(x)=\frac{d\mu(x)}{\log(e+d(x_{0},x))}.
  \end{equation}
\end{thm} 

This result is a generalization of Theorem A of \cite{BIJZ}. In Proposition \ref{bmo-hp}, we prove that the estimate is valid without weight for $\mathfrak b$ in $\bmo(\X,d,\mu)$, while in Theorem \ref{lipschitzhom} we obtain that Hardy-Orlicz class is replaced by the classical weight Hardy space $\H^{p}(\X,d,\tau)$ ($d\tau(x)=w(x)d\mu(x)$ for some appropriate weight) when $\mathfrak h\in\H^{p}(\X,d,\mu)$ and $\mathfrak b\in\Lambda_{\frac{1}{p}-1}(\X,d,\mu)$. This result is new even in the Euclidean case, since in \cite{BF} there was only a remark on the possibility of such estimate.

  Section 2 is devoted to notations and definitions. We recall in this paragraph the definition of spaces of homogeneous type and the grand maximal characterization of Hardy space as introduced in \cite{GLY}. In section 3, we give a prerequisite on Hardy-Orlicz space and prove some lemmas we need for our main result. We prove our main result in the last section, as well as its extensions. 

Throughout the paper, $C$ will denotes constants that are independent of the main parameters involved,  with values which may differ from line to line.

\section{Notations and definitions}

A quasimetric $d$ on a set $\X,$ is a function $d:\X\times
\X\rightarrow \left[ 0,\infty \right) $ which satisfies

\begin{enumerate}
\item[$\left( i\right) $] $d(x,y)=0$ if and only if $x=y$ ;

\item[$\left( ii\right) $] $d(x,y)=d(y,x)$ for all $x,y$ in $\X$;

\item[$\left( iii\right) $] there exists a finite constant $K_{0} \geq 1$
such that 
\begin{equation}
d(x,y)\leq K_{0} \left( d(x,z)+d(z,y)\right)  \label{0.001}
\end{equation}

for all $x,y,z$ in $\X.$
\end{enumerate}
The set $\X$ equipped with a quasimetric $d$ is called quasimetric space.
%
%Let $(\X,d)$ be a quasimetric space.
%
% Given $x\in \X$ \ and $r>0,$ the set $B( x,r)=\left\{
%y\in \X:d\left( x,y\right) <r\right\} $ is the ball  centered at $x$ and
%with radius $r$. For an arbitrary ball $B$ we denote $x_{B}$  its center and  $r(B)$ its radius. If $\delta$ is a positive real number and $B$ a ball, then  $\delta B$ denotes the ball having same center as $B$ but with radius $\delta$ times the radius of $B$. 

Let $\mu$ be a positive Borel measure on $(\X,d)$ such that all balls defined by $d$ have finite and positive measure. We say that the triple $(\X,d,\mu)$ is a space of homogeneous type if there exists a constant $C\geq 1$ such that for all $x\in\X$ and $r>0$, we have 
\begin{equation}
\mu(B(x,2r))\leq C\mu(B(x,r)).\label{doublingproperty}
\end{equation}
This property is known as the doubling property. If $C_{0}$ is the smallest constant for which (\ref{doublingproperty}) holds, then by iterating (\ref{doublingproperty}), we have
\begin{equation}
\frac{\mu(B)}{\mu(\tilde{B})}\leq C_{\mu}\left(\frac{r(B)}{r(\tilde{B})}\right)^{\texttt{n}}\text{ for all balls } \tilde{B}\subset B\label{doublingcondition}
\end{equation}
where $\n=\log_{2}(C_{0})$ and  $C_{\mu}=C_{0}(2K_{0})^{\n}$.
% Let $0<\kappa\leq n$. The triple $(\X,d,\mu)$ is called a $(\kappa,n)$-space if there exist constants $0<C_{1}\leq 1$ and $C_{2}\geq 1$ such that for all $x\in\X$, $0<r<\frac{\text{diam}(\X)}{2K_{0}}$ and $1\leq \lambda<\frac{\text{diam}(\X)}{2K_{0}r}$,
%\begin{equation}
%C_{1}\lambda^{\kappa}\mu(B(x,r))\leq\mu(B(x,\lambda r))\leq C_{2}\lambda^{n}\mu(B(x,r)),\label{RD}
%\end{equation}
%where $\text{diam}(\X)=\sup_{x,y\in\X}d(x,y)$.
%
%A space of homogeneous type is called and R.D-space, if it is a $(\kappa,n)$-space for some $0<\kappa\leq n$.

Notice that from the reverse doubling property, $\mu(\left\{x\right\})=0$  for all $x\in\X$. We also have that
	\begin{equation}
	\mu\left(B(x,r+d(x,y)\right)\sim \mu\left(B(y,r)\right)+\mu\left(B(y,d(y,x)\right)
	\end{equation}
	for $x,y\in\X$ and $r>0$.

In this paper, $\X=(\X,d,\mu)$ is a space of homogeneous type in which relations (\ref{reversedoubling}) and (\ref{C1}) are satisfy. We also assume (see \cite{MS1}) that there exist two constants $A'_{0}>0$ and $0<\theta\leq1$ such that 
\begin{equation}
\left|d(x,z)-d(y,z)\right|\leq A'_{0}d(x,y)^{\theta}\left[d(x,z)+d(y,z)\right]^{1-\theta}\label{order}.
\end{equation}
The space is saying to be of order $\theta$. We will refer to the constants  $K_{0},C_{0},\n,,\kappa, C_{\mu}$,$c_{\mu},A'_{0}$ and $\theta$ mentioned above, as the constants of the space. We will not mention the measure and the quasimetric when talking about the space $(\X,d,\mu)$. But if we use another measure than $\mu$, this will be mentioned explicitly. The following abbreviation for the measure of balls will be also used 
\begin{equation}
V_{r}(x)=\mu\left(B(x,r)\right) \text{ and } V(x,y)=\mu\left(B(x,d(x,y))\right),
\end{equation}
for all $x,y\in \X$ and $r>0$.

\begin{defn}\cite{GLY}
Let $x_{0}\in \X$, $r>0$, $0<\beta\leq 1$ and $\gamma>0$. A complex values function $\varphi$ on $\X$ is called a test function of type $(x_{0},r,\beta,\gamma)$ if the following hold:
\begin{enumerate}
	\item [(i)] $	\left|\varphi(x)\right|\leq C\frac{1}{\mu\left(B(x,r+d(x,x_{0})\right)}\left(\frac{r}{r+d(x_{0},x)}\right)^{\gamma}$ for all $x\in \X,$
	\item [(ii)] $\left|\varphi(x)-\varphi(y)\right|\leq C\left(\frac{d(x,y)}{r+d(x_{0},x)}\right)^{\beta}\frac{1}{\mu\left(B(x,r+d(x,x_{0})\right)}\left(\frac{r}{r+d(x_{0},x)}\right)^{\gamma}$	for all $x,y$ in $\X$ satisfying $d(x,y)\leq\frac{r+d(x_{0},x)}{2K_{0}}.$
\end{enumerate}
We denote by $\mathcal G (x_{0},r,\beta,\gamma)$ the set of all test functions of type $(x_{0},r,\beta,\gamma)$, equipped with the norm

\begin{equation}
\left\|\varphi\right\|_{\mathcal G (x_{0},r,\beta,\gamma)}=\inf\left\{C: (i) \text{ and } (ii) \text{ hold} \right\}.
\end{equation}
\end{defn}

\medskip

In the sequel, we will fix an element $x_{0}$ in $\X$ and put $\mathcal G (\beta,\gamma)=\mathcal G(x_{0},1,\beta,\gamma)$.
It is easy to prove that 
\begin{equation}
\mathcal G(x_{1},r,\beta,\gamma)=\mathcal G (\beta,\gamma),
\end{equation}
with equivalent norms for all $x_{1}\in\X$ and $r>0$. Furthermore, it is easy to check that $\mathcal G(\beta,\gamma)$ is a Banach space. 

\medskip

For a given $\epsilon\in\left(0,\theta\right]$ and $\beta,\gamma\in\left(0,\epsilon\right]$, $\mathcal G ^{\epsilon}_{0}(\beta,\gamma)$ denotes the completion of $\mathcal G (\epsilon,\epsilon)$ in $\mathcal G (\beta,\gamma)$. Equipp $\mathcal G ^{\epsilon}_{0}(\beta,\gamma)$ with the norm $\left\|\varphi\right\|_{\mathcal G^{\epsilon}_{0} (\beta,\gamma)}=\left\|\varphi\right\|_{\mathcal G (\beta,\gamma)}$, and denote $\left(\mathcal G^{\epsilon}_{0} (\beta,\gamma)\right)'$ its dual space; that is the set of linear functionals $f$ from  $\mathcal G^{\epsilon}_{0} (\beta,\gamma)$ to $\bC$ with the property that there exists a constant $C>0$ such that for all $\varphi\in\mathcal G^{\epsilon}_{0} (\beta,\gamma)$, $\left|\left\langle f,\varphi\right\rangle\right|\leq C\left\|\varphi\right\|_{\mathcal G (\beta,\gamma)}$. This dual space will be refer to as a distribution space.

For $f\in\left(\mathcal G^{\epsilon}_{0} (\beta,\gamma)\right)'$, the grand maximal function $f^{\ast}$ of $f$ in the sense of Grafakos, Liu and Yang \cite{GLY} is defined for $x\in\mathcal X$ by 
\begin{equation}
f^{\ast}(x)=\sup\left\{\left|\left\langle f,\varphi\right\rangle\right|:\varphi\in\mathcal G^{\epsilon}_{0} (\beta,\gamma),\left\|\varphi\right\|_{\mathcal G(x,r,\beta,\gamma)}\leq 1\text{ for some }r>0\right\}.
\end{equation}

 The corresponding Hardy space $\H^{p}(\X)$ is defined for $p\in\left(0,\infty\right]$ to be the set of $\mathfrak h\in\left(\mathcal G^{\epsilon}_{0}(\beta,\gamma)\right)'$ for which
\begin{equation}
\left\|\mathfrak h\right\|_{\H^{p}(\X)}:=\left\|\mathfrak h^{\ast}\right\|_{L^{p}(\X)}<\infty.\label{normhp}
\end{equation}

It is proved in Proposition 3.15 and Theorem 4.17 of \cite{GLY} that for $\epsilon\in\left(0,\theta\right]$ and $p\in\left(\frac{\n}{\n+\epsilon},1\right]$, the definition of $\H^{p}(\X)$ as stated above is independent of the choice of the underlying space of distribution, i. e. if $f\in(\mathcal G^{\epsilon}_{0}(\beta_{1},\gamma_{1}))'$ with 
\begin{equation}
\n(1/p-1)<\beta_{1},\gamma_{1}<\epsilon\label{dist}
\end{equation}
 and $\left\|\mathfrak h\right\|_{\H^{p}(\X)}<\infty$ then $f\in(\mathcal G^{\epsilon}_{0}(\beta_{2},\gamma_{2}))'$ for every $\beta_{2}$ and $\gamma_{2}$ satisfying (\ref{dist}). 
 
 In the rest of the paper $0<\epsilon\leq \theta$ is fixed and $p\in\left(\frac{\n}{\n+\epsilon},1\right]$. We also fix the underline space of distribution $\mathcal G^{\epsilon}_{0}(\beta,\gamma))'$ with $\beta$ and $\gamma$ as in (\ref{dist}).
 
 \medskip
 
As mentioned in the introduction, the dual space of $\H^{1}(\X)$ is $\BMO(\X)$(space of bounded mean oscillation function), defined as the set of locally integrable functions $\mathfrak b$ satisfying

\begin{equation}
\frac{1}{\mu(B)}\int_{B}\left|\mathfrak b(x)-\mathfrak b_{B}\right|d\mu(x)\leq A \label{bmo},\text{ for all ball } B,
\end{equation}
where $\mathfrak b_{B}=\frac{1}{\mu(B)}\int_{B}\mathfrak b(x)d\mu(x)$, and $A$ a constant depending only on $\mathfrak b$ and the space constant. We put 
\begin{equation}
\left\| \mathfrak b\right\|_{\BMO(\X)}=\sup_{B:ball}\frac{1}{\mu(B)}\int_{B}\left|\mathfrak b(x)-\mathfrak b_{B}\right|d\mu(x)
\end{equation}
and
\begin{equation}
\left\|\mathfrak b\right\|_{\BMO^{+}}=\left\|\mathfrak b\right\|_{\BMO(\X)}+\left|f_{\mathbb B}\right|,
\end{equation}

where $\mathbb B$ is the ball center at $x_{0}$ and with radius $1$. When the measure of $\X$ is finite, $\left(\BMO(\X), \left\| \cdot\right\|_{\BMO}\right)$ is a Banach space. The set of equivalence classes of functions under the relation ''$\mathfrak b_{1}$ and $\mathfrak b_{2}$ in $\BMO(\X)$ are equivalent if and only if $\mathfrak b_{1}-\mathfrak b_{2}$ is constant'' which we still denote by $\BMO(\X)$ equipped with $\left\|\cdot\right\|_{\BMO(\X)}$ is a Banach space .

As proved in \cite{CW2}, we have that for every $1\leq q<\infty$
	\begin{equation}\left\|\mathfrak b\right\|_{\BMO(\X)}\lsim\sup_{B:ball}\left(\frac{1}{\mu(B)}\int_{B}\left|\mathfrak b-\mathfrak b_{B}\right|^{q}d\mu\right)^{\frac{1}{q}} \lsim\left\|\mathfrak b\right\|_{\BMO(\X)},\label{equibmo}
	\end{equation}
	for all $\mathfrak b$ in $\BMO(\X)$, where the supremum is taken over all balls of $\X$.
	
	We also have by the doubling condition of the measure $\mu$, that for $\mathfrak b\in \BMO(\X)$, and $B$ a ball in $(\X,d)$, 
	
\begin{equation}
 \left|\mathfrak b_{B}-\mathfrak b_{2^{k}B}\right|\leq C (1+k)\left\|f\right\|_{\BMO(\X)} \text{ for all non negative integer }k,
\end{equation}

Theorem B of \cite{CW2} (see also Theorem 5.3 of \cite{HMY2}) stated that for $\frac{\n}{\n+\epsilon}<p<1$, the dual space of Hardy space $\H^{p}(\X)$ is the Lipschitz space $\Lambda_{\frac{1}{p}-1}(\X)$. We recall that for $0<\gamma$, the Lipschitz space $\Lambda_{\gamma}(\X)$ is the set of those functions $f$ on $\X$ for which 
\begin{equation}
\left| f(x)-f(y)\right|\leq A \mu\left(B\right)^{\gamma}, \label{lips}
\end{equation}
where $B$ is any ball containing both $x$ and $y$ and $A$ is a constant depending only on $f$.

We can see that this definition of Lipschitz recovers the Euclidean case only when $0<\gamma<\frac{1}{\n}$. In fact, unless $\gamma$ is sufficiently small, it can happen that the only functions satisfying (\ref{lips}) are the constants. But, as shown in \cite{CW1} there are situations where these spaces are not trivial. However, we are going to consider only $0<\gamma<\frac{\epsilon}{\n}$, since it is the range in which the atomic definition of Hardy coincides with the maximal function characterization. Let put
%Futhermore, as far as we are concerned in this note, we are going to  %For example if we consider $\X=\left[0,1\right)$, $\mu$ the Lebesgue measure and $d(x,y)$ the length of the smallest dyadic interval (consider closed in the left and open in the right) containing $x$ and $y$, any finite sum $\sum a_{j}\chi_{S_{j}}$, where $\chi_{S_{j}}$'s are characteristic functions of disjoint balls, belongs to $\Lambda_{\gamma}(\X)$ for all $\gamma>0$. Let put
\begin{equation}
\left\| f\right\|_{\Lambda_{\gamma(\X)}}=
\inf\left\{A:(\ref{lips}) \text{ holds}\right\} 
\end{equation}
then $\left\|\cdot\right\|_{\Lambda_{\gamma(\X)}}$ is a norm on the set of equivalence classes of functions under the relation ''$b_{1}$ and $ b_{2}$ in $\Lambda_{\gamma}(\X)$ are equivalent if and only if $b_{1}-b_{2}$ is constant'', which we still denote  $\Lambda_{\gamma}(\X)$.

\section{A prerequisite about Orlicz spaces}
Let  
\begin{equation} 
\wp(t)=\frac{t}{\log\left(e+t\right)}\text{ for all } t>0.
\end{equation}
A $\mu$-measurable function $f:\X\rightarrow\bR$ is said to belong to the Orlicz space $L^{\wp}(\X)$ if
\begin{equation}
\left\|f\right\|_{L^{\wp}}:=\inf\left\{k>0:\int_{\X}\wp\left(k^{-1}\left|f(x)\right|\right)d\mu(x)\leq 1\right\}<\infty.
\end{equation}
It is easy to see that $L^{1}(\X)\subset L^{\wp}(\X)$. More precisely, we have
\begin{equation}
\left\|f\right\|_{L^{\wp}(\X)}\leq \left\|f\right\|_{L^{1}(\X)}.
\end{equation} 
 We are going to recall some results involved Orlicz spaces mention in \cite{BIJZ}, which are also valid in the context of space of homogeneous type.
 
\begin{enumerate}
	\item [(i)]If $\ExpL(\X)$ is the Orlicz space associated to the Orlicz  function $t\mapsto e^{t}-1$ and $\LlogL(\X)$ the one associated to $t\mapsto t\log(e+t)$ then we have the following H\"older type inequality
	\begin{equation}
	\left\|fg\right\|_{L^{\wp}(\X)}\leq 4\left\|f\right\|_{L^{1}(\X)}\left\|g\right\|_{\ExpL(\X)}
	\end{equation}
	 for all $f\in L^{\wp}(\X)$ and $g\in\ExpL(\X)$ using the elementary inequality
	 \begin{equation}\label{ineq1}
	 \frac{ab}{\log(e+ab)}\leq a+e^{b}-1 \text{ for all } a,b\geq0.
	 \end{equation}
	 We also have the duality between $\ExpL(\X)$ and $\LlogL(\X)$, that is 
	 \begin{equation}
	 \left\|fg\right\|_{L^{1}(\X)}\leq 2\left\|f\right\|_{\LlogL(\X)}\left\|g\right\|_{\ExpL(\X)},
	 \end{equation}
	 using the following inequalities
	 \begin{equation}\label{ineq2}
	 ab\leq a\log(1+a)+e^{b}-1  \text{ for all } a,b\geq 0. 
	 \end{equation}
	 
	 \item[(ii)] Since the Orlicz function $\wp$ we consider is not convex, the triangular inequality does not hold for $\left\|\cdot\right\|_{L^{\wp}(\X)}$. But we have the following substitute 
	 \begin{equation}\label{additivity}
	 \left\|f+g\right\|_{L^{\wp}(\X)}\leq 4\left\|f\right\|_{L^{\wp}(\X)}+4\left\|g\right\|_{L^{\wp}(\X)},
	 \end{equation}
	for $f,g\in L^{\wp}(\X)$. This relation remain valid if we replace the measure $\mu$ by any one absolutely continuous compared to $\mu$.
	 \item [(iii)]$L^{\wp}(\X)$ equipped with the metric
\begin{equation}
\mathfrak d(f,g):=\inf\left\{\delta>0:\int_{\X}\wp\left(\delta^{-1}\left|f(x)-g(x)\right|\right)d\mu(x)\leq \delta\right\}
\end{equation}
 is a complete linear metric space.
 \item [(iv)] If $\mathfrak d(f,g)\leq 1,$ then  
 \begin{equation}
 \left\|f-g\right\|_{L^{\wp}}\leq \mathfrak d(f,g)\leq 1.\label{controle-norme-distance}
 \end{equation}
 \item [(v)]A sequence $\left(f_{n}\right)_{n>0}$ converge in $L^{\wp}(\X)$ to $f$ if and only if  $\lim_{n\rightarrow \infty}\left\|f_{n}-f\right\|_{L^{\wp}}=0$.
\end{enumerate}

We define the Hardy-Orlicz space $\H^{\wp}(\X)$, to be the subset of  $\mathcal G^{\epsilon}_{0}(\beta,\gamma)'$ consists of distributions $f$ such that $f^{\ast}\in L^{\wp}(\X)$, and we put
\begin{equation}
\left\|f\right\|_{H^ {\wp}(\X)}:=\left\|f^{\ast}\right\|_{L^
{\wp}(\X)}.
\end{equation}
In \cite{V}, it is proved that this characterization of Hardy-Orlicz spaces coincide with some atomic characterization.
\begin{lem}\label{controle}
Let $\mathfrak b$ be in $\BMO(\X)$ . There exists a constant $C$ such that for every $(1,q)$-atom $\a$ supported in a ball $B$,
\begin{equation}
\left\|(\mathfrak b-\mathfrak b_{B})\a^{\ast}\right\|_{L^{1}(\X)}\leq C\left\|\mathfrak b\right\|_{\BMO(\X)}.
\end{equation}
\end{lem}

\begin{proof}
Let $\mathfrak b\in \BMO(\X)$ and $\a$ a $(1,q)$-atom supported in $B=B_{(x_{0},R)}$. We have
\begin{equation}
\left\|\left(\mathfrak b-\mathfrak b_{B}\right)\mathfrak a^{\ast}\right\|_{L^{1}(\X)}=\int_{B_{(x_{0},2 K_{0}R)}}\left|\mathfrak b(z)-\mathfrak b_{B}\right|\mathfrak a^{\ast}(z)d\mu(z)+\int_{B^{c}(x_{0},2K_{0} R)}\left|\mathfrak b(z)-\mathfrak b_{B}\right|\mathfrak a^{\ast}(z)d\mu(z),\label{somme}
\end{equation}
where $B^{c}(x_{0},2K_{0}R)=\X\setminus B(x_{0},2 K_{0}R)$. Furthermore we have
\begin{equation}
\mathfrak a^{\ast}(z)\leq C \mathcal M\mathfrak a(z) \text{ for all } z\in \X,\label{max}
\end{equation}
where $\mathcal M\mathfrak a(z)=\sup_{B:B\ni z}\frac{1}{\mu(B)}\int_{B}\left|\a(x)\right|d\mu(x)$ denote the Hardy-Littlewood maximal function of $\a$,  according to Proposition 3.10 of \cite{GLY}. We also have
\begin{equation}
\a^{\ast}(z)\leq C\left(\frac{R}{d(z,x_{0})}\right)^{\beta}\frac{1}{\mu(B(z,d(z,x_{0})))},\text{ for all } z\notin B(x_{0},2K_{0} R),\label{maj}
\end{equation}

as it is shown in the proof of Lemma 4.4 of \cite{GLY}. If we take (\ref{max}) into first term of the sums (\ref{somme}) and use H$\ddot{\text{o}}$lder inequality with $1<q<\infty$, then we have
 
\begin{equation}
\int_{B_{(x_{0},2 K_{0}R)}}\left|\mathfrak b(z)-\mathfrak b_{B}\right|\a^{\ast}(z)d\mu(z)\leq\left(\int_{B(x_{0},2K_{0} R)}\left|\mathfrak b(z)-\mathfrak b_{B}\right|^{q'}d\mu(z)\right)^{\frac{1}{q'}}\left(\int_{\X}\mathcal M\a(z)^{q}d\mu(z)\right)^{\frac{1}{q}}.
\end{equation}

Since the Hardy-Littlewood maximal operator $\mathcal M$ is bounded in $L^{q}(\X)$, there exists a constant $C$ such that
\begin{equation}
\int_{B_{(x_{0},2 K_{0}R)}}\left|\mathfrak b(z)-\mathfrak b_{B}\right|\a^{\ast}(z)d\mu(z)\leq C\left\|\mathfrak b\right\|_{\BMO(\X)}
\end{equation}
 according to relation (\ref{equibmo}). 

On the other hand if we take (\ref{maj}) in the second term of (\ref{somme}) we have
\begin{equation}
\begin{aligned}
&\int_{B^{c}(x_{0},2 K_{0}R)}\left|\mathfrak b(z)-\mathfrak b_{B}\right|a^{\ast}(z)d\mu(z)\\
&\ \ \ \ \ \ \ \ \ \ \ \ \ \ \leq C\sum^{\infty}_{k=1}\int_{(2K^{0})^{k+1}B\setminus(2K_{0})^{k}B}\left(\frac{R}{d(z,x_{0})}\right)^{\beta}\frac{\left|\mathfrak b(z)-\mathfrak b_{B}\right|}{\mu(B(z,d(z,x_{0})))}d\mu(z)\\
&\ \ \ \ \ \ \ \ \ \ \ \ \ \ \leq C\sum^{\infty}_{k=1}(2K_{0})^{-k\beta}\left[\frac{1}{\mu((2K_{0})^{k+1}B)}\int_{(2K_{0})^{k+1}B}\left|\mathfrak b(z)-\mathfrak b_{(2K_{0})^{k+1}B}\right|d\mu(z)+\left|\mathfrak b_{(2K_{0})^{k+1}B}-\mathfrak b_{B}\right|\right],
\end{aligned}
\end{equation}
where the second inequality comes from the fact that $\mu(B(z,d(z,x_{0}))\sim\mu(B(x_{0},d(z,x_{0}))$.
Since the series $\sum^{\infty}_{k=1}(2K_{0})^{-k\beta}$ converges, we also have that there exists a constant $C$ not depending on $\mathfrak b$ and $\a$, such that
\begin{equation}
 \int_{B^{c}(x_{0},2 K_{0}R)}\left|\mathfrak b(z)-\mathfrak b_{B}\right|\a^{\ast}(z)d\mu(z)\leq C \left\|\mathfrak b\right\|_{\BMO(\X)},
 \end{equation}
  which end the proof. 
\end{proof} 

It is well known that the John-Nirenberg inequality is valid in the context of space of homogeneous type (see \cite{MMNO}). This inequality states that there exist constants $K_{1}$ and $K_{2}$ such that for any $\mathfrak b\in \BMO(\X)$ with $\left\|\mathfrak b\right\|_{\BMO(\X)}\neq 0$ and any ball $B\subset\X$, we have
\begin{equation}
\mu\left(\left\{x\in B:\left|\mathfrak b(x)-\mathfrak b_{B}\right|>\lambda\right\}\right)\leq K_{1}\exp\left(-\frac{K_{2}\lambda}{\left\|\mathfrak b\right\|_{\BMO(\X)}}\right)\mu(B) \text{ for all }\lambda>0.
\end{equation}
An immediate consequence of this inequality is that there is a constant $K_{3}$ depending only on the space constants, such that 
\begin{equation}
\frac{1}{\mu(B)}\int_{B}\exp\left(\frac{\left|\mathfrak b-\mathfrak b_{B}\right|}{K_{3}\left\|\mathfrak b\right\|_{\BMO(\X)}}\right)\leq 2.\label{jn}
\end{equation}
for all balls $B$ in $\X$ .

Notice that we can choose  $K_{3}$ as big as we like.
%\begin{lem} There exists a constant $C$ such that for every ball $B$ with radius 1, if $\mathfrak b\in \BMO(\X)$ and $\psi\in L^{1}(B)$, then
%\begin{equation}
%e^{-\left|\mathfrak b_{B}\right|}\int_{B}\frac{\left|\mathfrak b(x)\cdot \psi(x)\right|}{\log\left(e+\left|\mathfrak b(x)\cdot\psi(x)\right|\right)}d\mu(x)\leq C\left\|\mathfrak b\right\|_{\BMO(\X)}\int_{B}\left|\psi(x)\right|d\mu(x).
%\end{equation}
%\end{lem}

%
%
\begin{lem}\label{lemint}
Let $\mathbb B$ be the ball centered at $x_{0}$ with radius $1$. There exists a positive constant $K_{4}$  such that for any $\mathfrak b\in \BMO(\X)$ with $\left\|\mathfrak b\right\|_{\BMO(\X)}\neq 0$ we have 
\begin{equation}
\int_{X}\frac{e^{\frac{\left|\mathfrak b(x)-\mathfrak b_{\mathbb B}\right|}{K_{4}\left\|\mathfrak b\right\|_{\BMO(\X)}}}-1}{\left(1+d(x_{0},x)\right)^{2\n}}d\mu(x)\leq 1.
\end{equation}
\end{lem}

\begin{proof}
 Let $\mathfrak b\in\BMO(\X)$ with $\left\|\mathfrak b\right\|_{\BMO(\X)}\neq 0$. We have
\begin{equation}
\int_{X}\frac{e^{\frac{\left|\mathfrak b(x)-\mathfrak b_{\mathbb B}\right|}{K_{3}\left\|\mathfrak b\right\|_{\BMO(\X)}}}-1}{\left(1+d(x_{0},x)\right)^{2\n}}d\mu(x)=\int_{\mathbb B}\frac{e^{\frac{\left|\mathfrak b(x)-\mathfrak b_{\mathbb B}\right|}{K_{3}\left\|\mathfrak b\right\|_{\BMO(\X)}}}-1}{\left(1+d(x_{0},x)\right)^{2\n}}d\mu(x)+\int_{\mathbb B^{c}}\frac{e^{\frac{\left|\mathfrak b(x)-\mathfrak b_{\mathbb B}\right|}{K_{3}\left\|\mathfrak b\right\|_{\BMO(\X)}}}-1}{\left(1+d(x_{0},x)\right)^{2\n}}d\mu(x),
\end{equation}
where $\mathbb B^{c}=\X\setminus \mathbb B$. The first term in the right hand side is less that $\mu(\mathbb B)$. for the second term, we have
\begin{eqnarray*}
\int_{\mathbb B^{c}}\frac{e^{\frac{\left|\mathfrak b(x)-\mathfrak b_{\mathbb B}\right|}{K_{3}\left\|\mathfrak b\right\|_{\BMO(\X)}}}-1}{\left(1+d(x_{0},x)\right)^{2\n}}d\mu(x)&=&\sum^{\infty}_{k=0}\int_{2^{k}\leq d(x_{0},x)<2^{k+1}}\frac{e^{\frac{\left|\mathfrak b(x)-\mathfrak b_{\mathbb B}\right|}{K_{3}\left\|\mathfrak b\right\|_{\BMO(\X)}}}-1}{\left(1+d(x_{0},x)\right)^{2\n}}d\mu(x)\\
&\leq&\sum^{\infty}_{k=0}2^{-2\n k}\int_{B(x_{0},2^{k+1})}\left(e^{\frac{\left|\mathfrak b(x)-\mathfrak b_{\mathbb B}\right|}{K_{3}\left\|\mathfrak b\right\|_{\BMO(\X)}}}-1\right)d\mu(x).
\end{eqnarray*} 
 Using the fact that $\left|\mathfrak b_{\mathbb B}-\mathfrak b_{B(x_{0},2^{k+1})}\right|\leq \log(2^{\frac{C_{0}(k+1)}{\log 2}})\left\|\mathfrak b\right\|_{\BMO(\X)}$ and $\mu(B(x_{0},2^{k+1})\leq 2^{(k+1)\log_{2}C_{0}}\mu(\mathbb B)$, we have the term we are estimated less than
 \begin{equation}
 C\mu\left(\mathbb B\right)\sum^{\infty}_{k=0}2^{(-\n+\frac{C_{0}}{K_{3}\log2}) k}.\label{maj}
 \end{equation} 
 Take $K_{3}>\frac{C_{0}}{\n \log2}$. Then the series (\ref{maj}) converges. Therefore, 
 \begin{equation}
 \int_{X}\frac{e^{\frac{\left|\mathfrak b(x)-\mathfrak b_{\mathbb B}\right|}{K_{3}\left\|\mathfrak b\right\|_{\BMO(\X)}}}-1}{\left(1+d(x_{0},x)\right)^{2\n}}d\mu(x)\leq C\mu\left(\mathbb B\right).
 \end{equation}
 The result follows.
 
\end{proof}

\medskip

Let us introduce the following measures
\begin{equation}
d\nu:=\frac{d\mu(x)}{\log(e+d(x_{0},x))}\text{ and } d\sigma(x):=\frac{d\mu(x)}{(1+d(x_{0},x))^{2\n}},
\end{equation}
where $\n$ is the dimension of $\X$.
 It follows from the above lemma that for $\mathfrak b\in \BMO(\X)$ we have 
\begin{equation}
\left\|\mathfrak b-\mathfrak b_{\mathbb B}\right\|_{\ExpL(\X,\sigma)}\leq C\left\|\mathfrak b\right\|_{\BMO(\X)}.\label{bmo-exp}
\end{equation}
We can also see that for a $\nu$-measurable function $f$, we have
\begin{equation}\label{compare}
\left\|f\right\|_{L^{\wp}(\X,\nu)}\leq\left\|f\right\|_{L^{1}(\X)}.
\end{equation}

The next result is the analogous of Lemma 3.2 of \cite{BIJZ} in the context of spaces of homogeneous type, and its proof is just an adaptation of the one give in that paper. 

\begin{lem}\label{normp}
Let $f\in \ExpL (\X,\sigma)$. Then for $g\in L^{1}(\X)$ we have $g\cdot f\in L^{\wp}(\X,\nu)$ and 
\begin{equation}
\left\|g\cdot f\right\|_{L^{\wp}(\X,\nu)}\leq C\left\|g\right\|_{L^{1}(\X)}\left\|f\right\|_{\ExpL(\X,\sigma)}.\label{holder1}
\end{equation}
If moreover $f\in \BMO(\X)$ then 
\begin{equation}
\left\|g\cdot f\right\|_{L^{\wp}(\X,\nu)}\leq C\left\|g\right\|_{L^{1}(\X)}\left\|f\right\|_{\BMO^{+}(\X)}.\label{holder2}
\end{equation}

\end{lem}

\begin{proof}
Let $f\in \ExpL(\X,\sigma)$ and $g\in L^{1}(\X)$. If $\left\|g\right\|_{L^{1}(\X)}=0$ or $\left\|f\right\|_{\ExpL(\X,\sigma)}=0$ then there is nothing to prove. Thus we assume that $\left\|g\right\|_{L^{1}(\X)}\left\|f\right\|_{\ExpL(\X,\sigma)}\neq 0$.
 Let us put $A=8\n\left\|g\right\|_{L^{1}(\X)}$ and $B=8\n\left\|f\right\|_{\ExpL(\X,\sigma)}$. We are going to prove that the constant $C$ is $64\n^{2}$. For this it is sufficient to prove that

\begin{equation}
\int_{\X}\frac{\frac{1}{AB}\left|fg\right|
d\mu(x)}{\log\left(e+\frac{1}{AB}\left|fg\right|\right) \log(e+d(x_{0},x))}\leq 1.
\end{equation}

For this purpose, we will use the following elementary inequality :
\begin{equation}\label{loge-logex}
2\n\log(e+d(x_{0},x))>\log(e+(1+d(x_{0},x))^{2\n})\text{ for all } x\in\X,
\end{equation}
and for all  $a,b>0$,
\begin{equation}\label{productlog}
\log (e+a)\log(e+b)>\frac{1}{2}\log (e+ab).
\end{equation}  
It comes from the relation (\ref{loge-logex}) that 
\begin{equation}
\frac{\frac{1}{AB}\left|fg\right|
}{\log\left(e+\frac{1}{AB}\left|fg\right|\right) \log(e+d(x_{0},x))}\leq\frac{\frac{2\n}{AB}\left|fg\right|
}{\log\left(e+\frac{1}{AB}\left|fg\right|\right) \log(e+(1+d(x_{0},x))^{2\n})}
\end{equation}

so that applying relation (\ref{productlog}) to the left hand side of the inequality, yields

\begin{eqnarray*}
\frac{\frac{1}{AB}\left|fg\right|
}{\log\left(e+\frac{1}{AB}\left|fg\right|\right) \log(e+d(x_{0},x))}&\leq&\frac{\frac{4\n}{AB}\left|fg\right|
}{\log\left(e+\frac{1}{AB}\left|fg\right|(1+d(x_{0},x))^{2\n}\right)}\\
&\leq& 4\n\frac{\left|g\right|}{B}+\frac{4\n\left(e^{\frac{\left|f\right|}{A}}-1\right)}{(1+d(x_{0},x))^{2\n}},
\end{eqnarray*}

according to relation (\ref{ineq1}).
Taking the integral of both sides we obtain inequality (\ref{holder1}), since $$\frac{4\n\left(e^{\frac{\left|f\right|}{A}}-1\right)}{(1+d(x_{0},x))^{2\n}}\leq\frac{1}{2}\frac{\left(e^{8\n\frac{\left|f\right|}{A}}-1\right)}{(1+d(x_{0},x))^{2\n}}=\frac{1}{2}\frac{\left(e^{\frac{\left|f\right|}{\left\|f\right\|_{\ExpL(\X,\sigma)}}}-1\right)}{(1+d(x_{0},x))^{2\n}},$$
and 
$$ 4\n\frac{\left|g\right|}{B}=\frac{1}{2}\frac{\left|g\right|}{\left\|g\right\|_{L^{1}(\X)}}.$$
The inequality (\ref{holder2}) is also trivial if $\left\|f\right\|_{\BMO(\X)}=0$. 
Thus we assume that $f$ is not constant almost everywhere and we put  $f\cdot g=(f-f_{\mathbb B})\cdot g+f_{\mathbb B}\cdot g$, so that using relation (\ref{additivity}), relation (\ref{holder1}) and (\ref{bmo-exp}), we have
\begin{eqnarray*}
\left\|f\cdot g\right\|_{L^{\wp}(\X,\nu)}&\leq& C\left(\left\|(f-f_{\mathbb B})\cdot g\right\|_{L^{\wp}(\X,\nu)}+\left\|f_{\mathbb B}\cdot g\right\|_{L^{\wp}(\X,\nu)}\right)\\
&\leq& C\left(\left\|f-f_{\mathbb B}\right\|_{\ExpL(\X,\sigma)}\left\|g\right\|_{L^{1}(\X)}+\left|f_{\mathbb B}\right|\left\| g\right\|_{L^{1}(\X)}\right)\\
&\leq&C\left\| g\right\|_{L^{1}(\X,\nu)}\left\|f\right\|_{\BMO^{+}(\X)},
\end{eqnarray*}
which complete our proof.
\end{proof}

\section{ Proof of our main result} 
%  
%
%\begin{thm}\label{main1} 
%For  $\mathfrak b$ in $\BMO(\X)$ and $\mathfrak h$ in $\H^{1}(\X)$, we can give a meaning to the product $\mathfrak b\times \mathfrak h$ in the sense of distribution. Futhermore,  
%\begin{equation}
%\mathfrak b\times \mathfrak h\in L^{1}(\X)+\H^\wp (\X,\nu).
%\end{equation}
%\end{thm}
%
\begin{proof}[\textbf{Proof of Theorem \ref{main}}]   
Let $\mathfrak b\in \BMO(\X)$ and $\mathfrak h=\sum^{\infty}_{i=1}\lambda_{i}\a_{i}\in\H^{1}(\X)$, where $(\a_{i})_{i\geq 1}$ is a sequence of $(p,\infty)$-atoms, with $\a_{i}$ supported in the ball $B_{i}$, and $(\lambda_{i})_{i\geq 1}$ a sequence of scalars such that $\sum^{\infty}_{i=1}\left|\lambda_{i}\right|<\infty$. To prove our theorem, it is enough to show that the series
\begin{equation}
\sum^{\infty}_{i=1}\lambda_{i}\left(\mathfrak b-\mathfrak b_{B_{i}}\right)\a_{i}\text{ and }\sum^{\infty}_{i=1}\lambda_{j}\mathfrak b_{B_{i}}a_{i}
\end{equation}
are convergent in $L^{1}(\X)$ and $\H^{\wp}(X,\nu)$ respectively, since the product $\mathfrak b\times\mathfrak h$ by definition is the sum of both series.

The convergence of the first series in $L^{1}(\X)$ is immediate, since for all index  $i$ we have
 \begin{equation}
 \left\|\lambda_{i}\left(\mathfrak b-\mathfrak b_{B_{i}}\right)\a_{i}\right\|_{L^{1}(\X)}\leq \left|\lambda_{i}\right|\left\|\mathfrak b\right\|_{\BMO(\X)}\text{ and } \sum^{\infty}_{i=1}\left|\lambda_{i}\right|<\infty,
 \end{equation}
 according to Lemma \ref{controle}. For the second series, we consider the partial sum
\begin{equation}
S^{\ell}_{k}:=\sum^{\ell}_{i=k}\lambda_{i}\a_{i}\mathfrak b_{B_{i}}\label{partiel}\text{ for }k<\ell.
\end{equation} 
Our series converges in $\H^{\wp}(\X,\nu)$ if and only if 
$\lim_{k\rightarrow\infty}\left\|\left(S^{\ell}_{k}\right)^{\ast}\right\|_{L^{\wp}(\X,\nu)}=0$. But we have
\begin{equation*}
\left(S^{\ell}_{k}\right)^{\ast} \leq \sum^{\ell}_{i=k}\left|\lambda_{i}\right|\left(\a_{i}\mathfrak b_{B_{i}}\right)^{\ast}
\leq\sum^{\ell}_{i=k}\left|\lambda_{i}\right|\left|\mathfrak b-\mathfrak b_{B_{i}}\right|(\a_{i})^{\ast}+\left(\sum^{\ell}_{i=k}\left|\lambda_{i}\right|(\a_{i})^{\ast}\right)\left|\mathfrak b\right|,
\end{equation*}
so that
\begin{eqnarray*}
\left\|\left(S^{\ell}_{k}\right)^{\ast}\right\|_{L^{\wp}(\X,\nu)} &\leq& C\left[\left\|\sum^{\ell}_{j=k}\left|\lambda_{i}\right|\left|\mathfrak b-\mathfrak b_{B_{i}}\right|(\a_{i})^{\ast}\right\|_{L^{\wp}(\X,\nu)}+\left\|\left(\sum^{\ell}_{i=k}\left|\lambda_{i}\right|(\a_{i})^{\ast}\right)\left|\mathfrak b\right|\right\|_{L^{\wp}(\X,\nu)}\right]\\
&\leq& C\left[\left\|\sum^{\ell}_{i=k}\left|\lambda_{i}\right|\left|\mathfrak b-\mathfrak b_{B_{i}}\right|(\a_{i})^{\ast}\right\|_{L^{1}(\X)}+\left\|\left(\sum^{\ell}_{i=k}\left|\lambda_{i}\right|(\a_{i})^{\ast}\right)\left|\mathfrak b\right|\right\|_{L^{\wp}(\X,\nu)}\right]\\
&\leq& C\left\|\mathfrak b\right\|_{\BMO^{+}(\X)}\sum^{\ell}_{i=k}\left|\lambda_{i}\right|,
\end{eqnarray*}

where the last inequality come from Lemma \ref{controle} and  Lemma \ref{normp}.
It comes out that,
\begin{equation}
\lim_{k\rightarrow \infty}\left\|\left(S^{\ell}_{k}\right)^{\ast}\right\|_{L^{\wp}}\leq C\left\|\mathfrak b\right\|_{\BMO^{+}(\X)}\lim_{k\rightarrow\infty}\sum^{\ell}_{i=k}\left|\lambda_{i}\right|=0,
\end{equation}
since $\sum^{\infty}_{i=1}\left|\lambda_{i}\right|<\infty$.

 \end{proof}

 If we replace $BMO(\X)$ by $\bmo(\X)$, then we obtain that the Hardy-Orlicz space does not depend on a weight. More precisely, we obtain the following result
 \begin{prop}\label{bmo-hp}
 For  $\mathfrak b$ in $\bmo(\X)$ and $\mathfrak h$ in $\H^{1}(\X)$, we can give a meaning to the product $\mathfrak b\times \mathfrak h$ in the sense of distribution. Furthermore,  
\begin{equation}
\mathfrak b\times \mathfrak h\in L^{1}(\X)+\H^\wp (\X).
\end{equation}
 \end{prop}
 
 \begin{proof}
 The proof is almost similar to the one of Theorem \ref{main}. Let 
  $\mathfrak h\in\H^{1}(\X)$ be as in the previous theorem. We have for all $i$
 \begin{equation}
 \left\|(\mathfrak b-\mathfrak b_{B_{i}})\a_{i}\right\|_{L^{1}(\X)}\leq 2\left\|\mathfrak b\right\|_{\bmo(\X)},
 \end{equation}
 so that $\sum^{\infty}_{i=1}(\mathfrak b-\mathfrak b_{B_{i}})\a_{i}$ converge normally in $L^{1}(\X)$. 
 
 Since for all $i$ we have
 \begin{equation}
 \left(\mathfrak b_{B_{i}}\a_{i}\right)^{\ast}\leq \left|\mathfrak b-\mathfrak b_{B_{i}}\right|\a^{\ast}_{i}+\left|\mathfrak b\right|\a^{\ast}_{i},
 \end{equation}
 it follows that if 
 \begin{equation}
 \left|\mathfrak b\right|\left(\sum^{\infty}_{i=1}\lambda_{i}\a^{\ast}_{i}\right)
 \end{equation}
 belongs to $L^{\wp}(\X,d,\mu)$, then
 \begin{equation}
 \sum^{\infty}_{i=1}\lambda_{i}\mathfrak b_{B_{i}}\a_{i}
 \end{equation}
 converge in $\H^{\wp}(\X)$, since according to Lemma \ref{controle}, $\sum\lambda_{i}\left|\mathfrak b-\mathfrak b_{B_{i}}\right|\a^{\ast}_{i}$ converge normally in $L^{1}(\X)$ and therefore in $L^{\wp}(\X)$.
 Let us put $\psi=\left|\sum^{\infty}_{i=1}\lambda_{i}\a^{\ast}_{i}\right|\in L^{1}(\X)$, and consider a ball $B$ such that $\mu(B)=1$. We have, as proved in \cite{BF} that
 \begin{equation}
 \int_{B}\wp(\left|\mathfrak b\right|\psi)d\mu=\int_{B}\frac{\left|\mathfrak b\right|\psi}{\log(e+\left|\mathfrak b\right|\psi)}d\mu\leq C\left\|\mathfrak b\right\|_{\bmo(\X)}\int_{B}\psi d\mu.
 \end{equation}
 
 In fact, we have
 \begin{equation}
 \int_{B}\frac{\left|\mathfrak b\right|\psi}{\log(e+\left|\mathfrak b\right|\psi)}d\mu\leq \int_{B\cap\left\{\left|\mathfrak b\right|\leq1\right\}}\psi d\mu+\int_{B\cap\left\{\left|\mathfrak b\right|>1\right\}}\left|\mathfrak b\right|\frac{\psi}{\log(e+\psi)}d\mu.
 \end{equation}
 Since $\mathfrak b\in\bmo(\X,d,\mu)$ implies by the John-Nirenberg inequality (\ref{jn}) that there is a constant $C$ depending only on the space constant, such that $\left\|\mathfrak b\right\|_{\ExpL(B)}\leq C\left\|\mathfrak b\right\|_{\bmo(\X)}$ and $\left\|\frac{\psi}{\log(e+\psi)}\right\|_{\LlogL(B)}\leq\left\|\psi\right\|_{L^{1}(B)}$, the result follow from the duality between $\ExpL(B)$ and $\LlogL(B)$. 
 This being true for all ball $B$ of measure $1$, we take the sum over all such ball which are almost disjoint.
 
 \end{proof}
 
Let us consider now the Hardy space $\H^{p}(\X)$, with $p<1$. We have the following result

\begin{thm}\label{lipschitzhom}
Let $\frac{\n}{\n+1}<p<1$. For $f\in\Lambda_{\frac{1}{p}-1}(\X)$ and $g\in\H^{p}(\X)$ we can give a meaning to the product $f\times g$ as a distribution. Moreover, we have the inclusion
\begin{equation}
f\times g\in L^{1}(\X)+\H^{p}(\X,d,\tau),\text{ where }d\tau(x)=\frac{d\mu(x)}{\left(2K^{2}_{0}+K_{0}d(x_{0},x)\right)^{(1-p)\n}}.
\end{equation}
\end{thm}
\begin{proof}
Let $f\in\Lambda_{\frac{1}{p}-1}(\X)$ and $g\in\H^{p})$. We assume that $g$ has the following atomic decomposition
\begin{equation}
g=\sum^{\infty}_{i=1}\lambda_{i}\mathfrak a_{i},
\end{equation}
where $\mathfrak a_{i}'s$ are atoms supported respectively in the balls $B_{i}$. All we have to prove is that the series 
\begin{equation}
\sum^{\infty}_{i=0}\lambda_{i}(f-f_{B_{i}})\mathfrak a_{i}\label{bonne}
\end{equation}
and
\begin{equation}
\sum^{\infty}_{i=0}\lambda_{i}f_{B_{i}}\mathfrak a_{i}\label{mauvaise}
\end{equation}
converge respectively in $L^{1}(\X)$ and in $\H^{p}(\X,d,\tau).$
Arguing as in the previous theorem, we have that series (\ref{bonne}) converges normally in $L^{1}(\X)$. It remain to prove that (\ref{mauvaise}) converge in $\H^{p}(\X,d,\tau)$. As in Theorem \ref{main}, we have 
\begin{equation}
\left(S^{\ell}_{k}\right)^{\ast} \leq \sum^{\ell}_{i=k}\left|\lambda_{i}\right|\left(\a_{i}f_{B_{i}}\right)^{\ast}
\leq\sum^{\ell}_{i=k}\left|\lambda_{i}\right|\left|f-f_{B_{i}}\right|(\a_{i})^{\ast}+\left(\sum^{\ell}_{i=k}\left|\lambda_{i}\right|(\a_{i})^{\ast}\right)\left|f\right|,\label{controle1}
\end{equation} 
where $S^{\ell}_{k}=\sum^{\ell}_{i=k}\lambda_{i}\a_{i}\mathfrak b_{B_{i}}\text{ for }k<\ell.$ We claim that Lemma \ref{controle} remain true if we replace the space $\BMO(\X)$ by $\Lambda_{\frac{1}{p}-1}(\X)$ and the $(1,q)$-atoms by $(p,q)$-atoms $q\geq 1$, i.e. for $f\in\Lambda_{\frac{1}{p}-1}(\X)$ and $\mathfrak a$ a $(p,q)$-atom supported in the ball $B$,
\begin{equation}
\left\|(f-f_{B})\mathfrak a^{\ast}\right\|_{L^{1}}\leq C\left\|f\right\|_{\Lambda_{\frac{1}{p}-1}}.
\end{equation}
In fact, by the definition of Lipschitz space $\Lambda_{\frac{1}{p}-1}(\X)$, we have
\begin{equation}
\int_{B_{(x_{0},2 K_{0}R)}}\left|f(z)-f_{B}\right|\a^{\ast}(z)d\mu(z)\leq C\left\|f\right\|_{\Lambda_{\frac{1}{p}-1}(\X)}.
\end{equation}
In other respect 
\begin{equation}
\a^{\ast}(z)\leq C\mu\left(B(x_{0},R)\right)^{1-\frac{1}{p}}\left(\frac{R}{d(z,x_{0})}\right)^{\beta}\frac{1}{\mu(B(z,d(z,x_{0})))},
\end{equation}
for all $z\notin B(x_{0},2K_{0} R)$ according to Lemma 4.4 of \cite{GLY}.

 Arguing as in the proof of Lemma \ref{controle}, we have that $\sum\left|\lambda_{i}\right|\left|f-f_{B_{i}}\right|(\a_{i})^{\ast}$  converges in $L^{1}(\X)$. The proof of the Theorem will be complete if we establish that for any ball $B$ of radius $1$, we have for $f\in\Lambda_{\frac{1}{p}-1}(\X)$ and  $\psi\in L^{p}(B)$
\begin{equation}
\int_{B}(\left|f(x)\psi(x)\right|)^{p}d\tau(x)\leq C\left\|f\right\|^{p}_{\Lambda^{+}_{\frac{1}{p}-1}(\X)}\int_{B}\left|\psi(x)\right|^{p}d\mu(x),
\end{equation}
where $\left\|f\right\|^{p}_{\Lambda^{+}_{\frac{1}{p}-1}(\X)}=\left\|f\right\|^{p}_{\Lambda_{\frac{1}{p}-1}(\X)}+\max(\left|f(x_{0})\right|,1)^{p}$. 
Following the method in \cite{BF}, we have
 \begin{eqnarray*} \int_{B}\frac{\left|f(x)\psi(x)\right|^{p}}{(2K^{2}_{0}+K_{0}d(x_{0},x))^{\n(1-p)}}d\mu(x)&\leq& \int_{B\cap\left\{\left|f\right|\leq1\right\}}\left|\psi(x)\right|^{p} d\mu(x)\\ &+&\int_{B\cap\left\{\left|f\right|>1\right\}}\left|f(x)\right|^{p}\frac{\left|\psi(x)\right|^{p}}{(2K^{2}_{0}+K_{0}d(x_{0},x))^{\n(1-p)}}d\mu(x).
 \end{eqnarray*}
 Furthermore,
 \begin{equation}
 \begin{aligned} &\int_{B\cap\left\{\left|f\right|>1\right\}}\left|f(x)\right|^{p}\frac{\left|\psi(x)\right|^{p}}{(2K^{2}_{0}+K_{0}d(x_{0},x))^{\n(1-p)}}d\mu(x)\\ 
 &\ \ \ \ \ \ \ \ \ \ \ \ \ \ \ \ \ \ \ \ \leq\int_{B\cap\left\{\left|f\right|>1\right\}}\left|f(x)-f(x_{0}\right|^{p}\frac{\left|\psi(x)\right|^{p}}{(2K^{2}_{0}+K_{0}d(x_{0},x))^{\n(1-p)}}d\mu(x)\\ 
 &\ \ \ \ \ \ \ \ \ \ \ \ \ \ \ \ \ \ \ \ +\left|f(x_{0})\right|^{p}\int_{B\cap\left\{\left|f\right|>1\right\}}\frac{\left|\psi(x)\right|^{p}}{(2K^{2}_{0}+K_{0}d(x_{0},x))^{\n(1-p)}}d\mu(x).
 \end{aligned}
 \end{equation}
 Since $B\subset B(x_{0},2K^{2}_{0}+K_{0}d(x,x_{0}))$ for all $x$ in the ball $B$ of radius $1$, it comes from the definition of Lipschitz space $\Lambda_{\frac{1}{p}-1}(\X)$ that the first term in te right hand side of the above inequality is less or equal to
 \begin{equation*}
\left\|f\right\|^{p}_{\Lambda_{\frac{1}{p}-1}(\X)}\int_{B}\frac{\mu(B(x_{0},2K^{2}_{0}+K_{0}d(x_{0},x)))^{1-p}}{(2K^{2}_{0}+K_{0}d(x_{0},x))^{\n(1-p)}}\left|\psi(x)\right|^{p}d\mu(x).
 \end{equation*}
 But, from (\ref{C1}) and (\ref{normsup}) we have that $\mu(B(x_{0},2K^{2}_{0}+K_{0}d(x_{0},x)))\lsim\left(2K^{2}_{0}+K_{0}d(x_{0},x))\right)^{\n}$. 

Thus
\begin{equation}
\int_{B\cap\left\{\left|f\right|>1\right\}}\left|f(x)\right|^{p}\frac{\left|\psi(x)\right|^{p}}{(2K^{2}_{0}+K_{0}d(x_{0},x))^{\n(1-p)}}d\mu\lsim\left(\left\|f\right\|^{p}_{\Lambda_{\frac{1}{p}-1}(\X)}+\left|f(x_{0})\right|^{p}\right)\int_{B}\left|\psi(x)\right|d\mu(x).
\end{equation}
The result follow by covering the hold space by almost disjoint balls of radius $1$.
\end{proof}

\begin{remark}
Let $\frac{\n}{\n+\epsilon} <p<1$ and  $\gamma:=\frac 1p-1$. Then, for $h\in\mathcal H^p(\X)$ and $f\in\Lambda_{\gamma}(\X)\cap L^{\infty}(\X)$,
the product $h\times f$ can be given a meaning in the sense of
distributions. Moreover, we have the inclusion
\begin{equation}\label{inclusion}
  h\times f\in L^1(\X)+ \mathcal H^p(\X).
  \end{equation}
\end{remark} 

\begin{proof} Let $\mathfrak h\in\mathcal H^p(\X)$ be as in (\ref{atomdecompo}), where the atoms involved are $(p,\infty)$-atoms, and $f\in\Lambda_{\gamma}(\X)$. From Theorem \ref{lipschitzhom}, we have that
 \begin{equation}
\sum^{\infty}_{i=1}\lambda_{i}\left(f-f_{B_{i}}\right)\a_{i} 
 \end{equation}
 converge in $L^{1}(\X)$. For the series $\sum^{\infty}_{i=1}\lambda_{i}f_{B_{i}}\a_{i}$, we just have to remark that the functions $\frac{1}{\left\|f\right\|_{L^{\infty}(\X)}}f_{B_{j}}\a_{i}$ are $(p,\infty)$-atoms. In fact,
 	
\begin{enumerate}
	\item [(i)] supp$f_{B}\a_{i}\subset B_{i}$, since supp$\a_{i}\subset B_{i}$
	\item[(ii)] $\int_{\X} f_{B_{i}}\a_{i}(x)dx=0$
	\item[(iii)] $\left|f_{B_{i}}\a_{i}(x)\right|\leq  \left\|f\right\|_{L^{\infty}(\X)}\mu(B_{i})^{-\frac{1}{p}}$ 
\end{enumerate}
 and this end the proof, since $\sum^{\infty}_{i=1}\left|\lambda_{i}\right|^{p}<\infty$
\end{proof}
\begin{remark}
In the case $\mu(X)<\infty$, all our results remain valid, provided we consider the constant function $\mu(\X)^{-\frac{1}{p}}$ as an atom, and put 
\begin{equation}
\left\|\mathfrak b\right\|_{\BMO(\X)}=\sup_{B:ball}\frac{1}{\mu(B)}\int_{B}\left|\mathfrak b(x)-\mathfrak b_{B}\right|d\mu(x)+\left\|\mathfrak b\right\|_{L^{1}(\X)}
\end{equation}
and
\begin{equation}
\left\|f\right\|_{\Lambda_{\gamma}(\X)}=\sup\left\{\frac{\left|f(x)-f(y)\right|}{\mu(B)},\text{ for all ball } B\ni x,y\right\}+\left|\int_{\X}f(x)d\mu(x)\right|.
\end{equation}
In this case the reverse doubling condition (\ref{reversedoubling}), need to be satisfied just for small balls.
\end{remark}

\end{document}